\documentclass{article}
\usepackage[a4paper, total={7in, 10in}]{geometry}
\usepackage[T1]{fontenc}
\usepackage{graphicx}
\usepackage{mathtools}
\usepackage{amsmath,amsthm,amssymb,mathrsfs,amstext, titlesec,enumitem, comment, graphicx, color, xcolor, stmaryrd,mathabx}
\usepackage{thmtools}
\usepackage{xpatch}
\usepackage{tikz-cd}
\usepackage{nameref}
\usepackage{hyperref}
\makeatletter
\def\Ddots{\mathinner{\mkern1mu\raise\p@
\vbox{\kern7\p@\hbox{.}}\mkern2mu
\raise4\p@\hbox{.}\mkern2mu\raise7\p@\hbox{.}\mkern1mu}}
\makeatother
\newtheorem{theorem}{Theorem}[section]
\newtheorem{corollary}[theorem]{Corollary}

\theoremstyle{definition}
\newtheorem{definition}[theorem]{Definition}

\begin{document}
	
	\title{  Cartesian product of combinatorially rich sets- algebraic, elementary and dynamical approaches}

	\date{}
	\author{Pintu Debnath
		\footnote{Department of Mathematics,
			Basirhat College,
			Basirhat-743412, North 24th parganas, West Bengal, India.\hfill\break
			{\tt pintumath1989@gmail.com}}
	}
	\maketitle

\begin{abstract}
  Using the methods of topological dynamics,  H. Furstenberg introduced the notion of a central set and proved the famous Central Sets Theorem. D. De, N. Hindman, and D. Strauss introduced $C$-set,
satisfying the strong central set theorem. Using the algebraic structure
of the Stone-\v{C}ech compactification of a discrete semigroup, N.
Hindman and D. Strauss proved that the Cartesian product of two $C$-sets is
a $C$-set.  S. Goswami has proved the same result
using the elementary characterization of $C$-sets. In this article, we will prove
that the product of two $C$-sets is a $C$-set, using the dynamical characterization
of $C$-sets. Recently, S. Goswami has proved that the Cartesian product of two $CR$-sets is a $CR$-set, which was a question posed by  N. Hindman, H. Hosseini, D. Strauss, and M. Tootkaboni in [Semigroup Forum 107 (2023)]. Here we also prove that the Cartesian product of two essential $CR$-sets is an essential $CR$-set.
\end{abstract}

\textbf{Keywords:}  $CR$-set, $J$-set, $C$-set,
	 Algebra of the  Stone-\v{C}ech compactifications of discrete semigroups,  Dynamical systems.
	
	\textbf{MSC 2020:} 05D10, 22A15, 54D35.

\section{Introduction}
The notion of a central set was first discovered by H. Furstenberg using topological dynamics in \cite[Proposition 8.21 Page-169]{F}. Central sets are extremely important Ramsey theoretic objects, as whenever $\mathbb{N}$ is partitioned into finitely many cells,  one of them becomes central. They are also rich in combinatorial properties. Later in \cite{BH90}, V. Bergelson and N. Hindman proved an analog description of central sets using the algebra of the  Stone-\v{C}ech compactifications of discrete semigroups. Generalizing the definition of a central set, many notions of large sets were introduced,  which are also rich in combinatorial properties, for example quasi central set, $C$-set, $CR$-set, etc. In the present work, we shall investigate the presentation of their properties in the product.

We give a brief review of the algebraic structure of the Stone-\v{C}ech
compactification of discrete semigroups.

Let $S$ be a discrete semigroup. The elements of $\beta S$ are regarded as ultrafilters on $S$. Let $\overline{A}=\left\{p\in \beta S:A\in p\right\}$. The set $\{\overline{A}:A\subset S\}$ is a basis for the closed sets
of $\beta S$. The operation `$\cdot$' on $S$ can be extended to
the Stone-\v{C}ech compactification $\beta S$ of $S$ so that $(\beta S,\cdot)$
is a compact right topological semigroup (meaning that for each    $p\in\beta S$  the function $\rho_{p}\left(q\right):\beta S\rightarrow\beta S$ defined by $\rho_{p}\left(q\right)=q\cdot p$ 
is continuous) with $S$ contained in its topological center (meaning
that for any $x\in S$, the function $\lambda_{x}:\beta S\rightarrow\beta S$
defined by $\lambda_{x}(q)=x\cdot q$ is continuous). This is a famous
theorem due to Ellis that if $S$ is a compact right topological semigroup
then the set of idempotents $E\left(S\right)\neq\emptyset$. A nonempty
subset $I$ of a semigroup $T$ is called a $\textit{left ideal}$
of $S$ if $TI\subset I$, a $\textit{right ideal}$ if $IT\subset I$,
and a $\textit{two sided ideal}$ (or simply an $\textit{ideal}$)
if it is both a left and a right ideals. A minimal left ideal
is a left ideal that does not contain any proper left ideal. Similarly,
we can define a minimal right ideal and the smallest ideal.

Any compact Hausdorff right topological semigroup $T$ has the smallest
two sided ideal

$$
\begin{aligned}
	K(T) & =  \bigcup\{L:L\text{ is a minimal left ideal of }T\}\\
	&=  \bigcup\{R:R\text{ is a minimal right ideal of }T\}.
\end{aligned}$$

Given a minimal left ideal $L$ and a minimal right ideal $R$, $L\cap R$
is a group, and in particular contains an idempotent. If $p$ and
$q$ are idempotents in $T$; we write $p\leq q$ if and only if $pq=qp=p$.
An idempotent is minimal with respect to this relation if and only
if it is a member of the smallest ideal $K(T)$ of $T$. Given $p,q\in\beta S$
and $A\subseteq S$, $A\in p\cdot q$, if and only if the set $\{x\in S:x^{-1}A\in q\}\in p$,
where $x^{-1}A=\{y\in S:x\cdot y\in A\}$. See \cite{HS12} for
an elementary introduction to the algebra of $\beta S$ and any
unfamiliar details.

Let $A$ be a subset of $S$. Then $A$ is called central if and only if $A\in p$, for some minimal idempotent $p$ of $\beta S$.

A collection $\mathcal{F}\in\mathcal{P}(S)\setminus\left\{ \emptyset\right\} $
is upward hereditary if whenever $A\in\mathcal{F}$ and $A\subseteq B\subseteq S$
then it follows that $B\in\mathcal{F}$. A nonempty and upward hereditary
collection $\mathcal{F}\in\mathcal{P}(S)\setminus\left\{ \emptyset\right\} $
will be called a family. If $\mathcal{F}$ is a family, the dual family
$\mathcal{F}^{*}$ is given by,
\[
\mathcal{F}^{*}=\{E\subseteq S:\forall A\in\mathcal{F},E\cap A\neq\emptyset.\}
\]
A family $\mathcal{F}$ possesses the Ramsey property if whenever
$A\in\mathcal{F}$ and $A=A_{1}\cup A_{2}$ there is some $i\in\left\{ 1,2\right\} $
such that $A_{i}\in\mathcal{F}$.

\begin{definition}
    Let $\left(S,\cdot\right)$ be a semigroup. A family $\mathcal{F}$ is called left shift invariant iff for all $s\in S$ and for all $A\subseteq S$, $sA\in\mathcal{F}$, where $sA=\left\{sa:a\in A\right\}$.
\end{definition}

Let $\left(S,\cdot\right)$ be a discrete semigroup. Let $\mathcal{F}$ be a Ramsey family. Then by \cite[Theorem 3.11 Page-59]{HS12}, $\beta\left(\mathcal{F}\right)\neq\emptyset$,  where $\beta\left(\mathcal{F}\right)=\left\{p\in\beta S:\, p\subseteq\mathcal{F}\right\}$.  For a Ramsey family $\mathcal{F}$, if $\beta\left(\mathcal{F}\right)$ is subsemigroup of $\beta S$, then $E\left(\beta\left(\mathcal{F}\right)\right)\neq\emptyset$. If $\mathcal{F}$ is  Ramsey and shift invariant $\beta\left(\mathcal{F}\right)$ is a left ideal of $\beta S$ by \cite[Theorem 5.1.2]{C}.

\begin{definition}\label{def essential}\textbf{(Essential $\mathcal{F}$-set)}:
     Let $\mathcal{F}$ be a Ramsey family with  $\beta\left(\mathcal{F}\right)$ is a subsemigroup of $\beta S$. Then $A\subseteq S$ is called essential $\mathcal{F}$-set iff $\overline{A}\cap E\left(\beta\left(\mathcal{F}\right)\right)\neq\emptyset$.
\end{definition}

Let $\left(S,\cdot\right)$ be a semigroup and $A$ be a subset of $S$. Then for $s\in S$, we define $s^{-1}A=\left\{t:st\in S \right\}$.

\begin{definition}
Let $\left(S,\cdot\right)$ be a semigroup and $A\subseteq S$.

\begin{enumerate}
\item  The set $A$  is $IP$ set if and only if there exists a sequence
$\langle x_{n}\rangle _{n=1}^{\infty}$ in $S$ such that $FP\left(\langle  x_{n}\rangle _{n=1}^{\infty}\right)\subseteq A$.
Where 
\[
FP\left(\langle x_{n}\rangle _{n=1}^{\infty}\right)=\left\{ \prod_{n\in F}x_{n}:F\in\mathcal{P}_{f}\left(\mathbb{N}\right)\right\} 
\]
 and $\prod_{n\in F}x_{n}$ is the product in increasing order.
\item  The set $A$  is piecewise syndetic if and only if 
\[
\left(\exists H\in\mathcal{P}_{f}\left(S\right)\right)\left(\forall F\in\mathcal{P}_{f}\left(S\right)\right)\left(\exists x\in S\right)\left(Fx\subseteq\bigcup_{t\in H}t^{-1}A\right).
\]
\item  Let $l\in\mathbb{N}$. The set $A$ is a length $l$ progression if and only if
there exist , $a=\left(a\left(1\right),a\left(2\right)\right)\in S\times S$ and $d\in S$ such that 
\[
A=\left\{ a\left(1\right)d^{t}a\left(2\right):t\in\left\{ 1,2,\ldots,l\right\} \right\} .
\]
\item The set $A$ is a Prog-set if and only if for each $l\in\mathbb{N}$, $A$ contains
a length $l$ progression.

\item Let $l\in\mathbb{N}$. The
 set $B$ is a length $l$ weak progression if and only
if there exist $m\in\mathbb{N}$, $a\in S^{m+1}$ and $d\in S$ such
that 
\[
A=\left\{ a\left(1\right)d^{t}a\left(2\right)d^{t}\cdots a\left(m\right)d^{t}a\left(m+1\right):t\in\left\{ 1,2,\ldots,l\right\} \right\}.
\]
\item The set $A$ is a weak Prog(wProg) set if and only if for each $l\in\mathbb{N}$,
$A$ contains a length $l$ weak progression.
\item The set  $A$ is a $J$-set if and only if for $F\in\mathcal{P}_{f}\left(^{\mathbb{N}}S\right)$,
there exist $m\in\mathbb{N}$, $a\in S^{m+1}$, and $t\left(1\right)<t\left(2\right)<\ldots<t\left(m\right)$
in $\mathbb{N}$ such that for each $f\in F$, 
\[
a\left(1\right)f\left(t\left(1\right)\right)a\left(2\right)\cdots a\left(m\right)f\left(t\left(m\right)\right)a\left(m+1\right)\in A.
\]

\item Let $k\in\mathbb{N}$. The set  $A$ is a $k$-$CR$-set if and only if for $F\in\mathcal{P}_{f}\left(^{\mathbb{N}}S\right)$ with $|F|\leq k$,
there exist $m\in\mathbb{N}$, $a\in S^{m+1}$, and $t\left(1\right)<t\left(2\right)<\ldots<t\left(m\right)$
in $\mathbb{N}$ such that for each $f\in F$, 
\[
a\left(1\right)f\left(t\left(1\right)\right)a\left(2\right)\cdots a\left(m\right)f\left(t\left(m\right)\right)a\left(m+1\right)\in A.
\]

\item The set $A$ is $CR$-set if and only if $A$ is $k$-$CR$-set for all $k\in\mathbb{N}$.

\end{enumerate}
\end{definition}

Let $S$ and $T$ be semigroups, let $A$ be a $J$-set in $S$ and
$B$ be a $J$-set in $T$. Then,  from  \cite[Theorem 2.11]{G20},
we get $A\times B$ is a $J$-set in $S\times T$. We also know from \cite[Theorem 2.4]{G20}
that the product of two piecewise syndetic sets is a piecewise syndetic set.
It is trivial from the definition of Prog-set that the product of
two Prog-set is also a Prog-sets. But we do not know whether the product
of two wProg-sets is a wProg-set or not. Recently in \cite{G24}, S. Goswami proved that the product of two $CR$-sets is a $CR$-set. Let $S$ be a semigroup. We abbreviate 
\begin{itemize}
    \item the family of piecewise syndetic sets in $S$ as $\mathcal{PS}_{S}$,

    \item the family of infinite sets in $S$ as $\mathcal{IF}_{S}$,

    \item the family of Prog-sets in $S$ as $\mathcal{P}_{S}$,

    \item the family of wProg-sets in $S$ as $\mathcal{WP}_{S}$,

    \item the family of $J$-sets in $S$ as $\mathcal{J}_{S}$ and

    \item the family of $CR$-sets in $S$ as $\mathcal{CR}_{S}$.
    
\end{itemize}

 It is well known that $\mathcal{PS}_{S}$ is a Ramsey family, and it is obvious that the family $\mathcal{IF}_{S}$ is a Ramsey family. The remaining families  $\mathcal{P}_{S}$, $\mathcal{WP}_{S}$, $\mathcal{J}_{S}$ and $\mathcal{CR}_{S}$ are Ramsey families by \cite[Lemma 2.20]{H20}, \cite[Lemma 2.20]{H20}, \cite[Lemma 14.14.6 Page-345]{HS12} and \cite[Theorem 2.4]{HHST} respectively.

 As all the families $\mathcal{IF}_{S}$, $\mathcal{PS}_{S}$, $\mathcal{P}_{S}$, $\mathcal{J}_{S}$, and $\mathcal{CR}_{S}$ are Ramsey and left shift invariant,  $\beta\left(\mathcal{IF}_{S}\right)$,
$\beta\left(\mathcal{PS}_{S}\right)$, $\beta\left(\mathcal{P}_{S}\right)$,
$\beta\left(\mathcal{J}_{S}\right)$ and $\beta\left(\mathcal{CR}_{S}\right)$
are closed subsemigroups of $\beta S$.

From the  Definition \ref{def essential} together with the abbreviations, we get
quasi central set is an essential $\mathcal{PS}$-set and $C$-set
is an essential $\mathcal{J}$-set. In the remainder  of this article, we mean $\mathcal{F}_{S}$ by any one of  $\mathcal{IF}_{S}$, $\mathcal{PS}_{S}$, $\mathcal{P}_{S}$, $\mathcal{J}_{S}$ and $\mathcal{CR}_{S}$. Let $\left(S,\cdot\right)$ and $\left(T,\cdot\right)$ be semigroups. Let $A$ be an essential $\mathcal{F}_{S}$-set and $B$ be an essential $\mathcal{F}_{T}$-set.  In this article, we prove that the
Cartesian product $A\times B$ is also an essential $\mathcal{F}_{S\times T}$-set.
 We use algebraic, dynamical, and elementary techniques to prove our desired result.

\section{Algebraic  and elementary proof}

We start this section with the following theorem.

\begin{theorem}\label{essential algeb}
Let $S$ be a discrete semigroup and for every family $\mathcal{F\subseteq P}\left(S\right)$
with the Ramsey property, $\beta\left(\mathcal{F}\right)\subseteq\beta S$. Then the set
$A$ is a $\mathcal{F}$-set if and only if $\overline{A}\cap\beta\left(\mathcal{F}\right)\neq\emptyset$.
\end{theorem}

\begin{proof}
As $\mathcal{F}$ is partition regular, the theorem follows from
\cite[Theorem 3.11 Page-59]{HS12}.
\end{proof}
From the above theorem together with the fact that $\beta\left(\mathcal{F}_{S}\right)$ is a closed subsemigroup of $\beta S$, we get the following result.

\begin{corollary}
Let $\left(S,\cdot\right)$ be a  discrete semigroup and $A\subseteq S$.
    \begin{itemize}
     \item [(a)] The set  $A$ is a Prog-set if and only if $\overline{A}\cap\beta\left(\mathcal{P}_{S}\right)\neq\emptyset$.

    \item [(b)] The set $A$ is an essential $\mathcal{P}_{S}$-set if and only if $\overline{A}\cap E\left(\beta\left(\mathcal{P}_{S}\right)\right)\neq\emptyset$.

    \item[(c)] The set $A$ is infinite set if and only if  $\overline{A}\cap\beta\left(\mathcal{IF}_{S}\right)=\overline{A}\cap\left(\beta S\setminus S\right)\neq\emptyset$.

   \item[(d)] The set $A$ is $IP$-set if and only if  $\overline{A}\cap E\left(\beta\left(\mathcal{IF}_{S}\right)\right)=\overline{A}\cap E\left(\beta S\right)\neq\emptyset$.

   \item[(e)] The set $A$ is a $PS$-set if and only if $\overline{A}\cap\beta\left(\mathcal{PS}_{S}\right)=\overline{A}\cap\overline{K\left(\beta S\right)}\neq\emptyset$.

    \item[(f)] The set $A$ is a quasi-central set if and only if  $\overline{A}\cap E\left(\beta\left(\mathcal{PS}_{S}\right)\right)=\overline{A}\cap E\left(\overline{K\left(\beta S\right)}\right)\neq\emptyset$.

    \item[(g)] The set $A$ is $J$-set if and only if  $\overline{A}\cap\beta\left(\mathcal{J}_{S}\right)=\overline{A}\cap J\left(S\right)\neq\emptyset$.

    \item[(h)] The set $A$ is $C$-set if and only if  $\overline{A}\cap E\left(\beta\left(\mathcal{J}_{S}\right)\right)=\overline{A}\cap  E\left( J\left(S\right)\right)\neq\emptyset$.

    \item[(i)] The set $A$ is $CR$-set if and only if  $\overline{A}\cap\beta\left(\mathcal{CR}_{S}\right)=\overline{A}\cap CR\left(S\right)\neq\emptyset$.

   \item[(j)] The set $A$ is an essential $CR$-set if and only if  $\overline{A}\cap E\left(\beta\left(\mathcal{CR}_{S}\right)\right)=\overline{A}\cap E\left( CR\left(S\right)\right)\neq\emptyset$.

\end{itemize}

\end{corollary}

   Now, we provide an algebraic proof of our main result.

\begin{theorem}\label{main theorem}
Let $\left(S,\cdot\right)$ and $\left(T,\cdot\right)$ be discrete semigroups, 
let $A$ be an essential $\mathcal{F}_{S}$-set and $B$ be an essential
$\mathcal{F}_{T}$-set. Then $A\times B$ is an essential $\mathcal{F}_{S\times T}$-set.
\end{theorem}

\begin{proof}
Pick idempotents $p\in\beta\left(\mathcal{F}_{S}\right)$ and $q\in\beta\left(\mathcal{F}_{T}\right)$
such that $A\in p$ and $B\in q$.  As $A\times B$ is a $\mathcal{F}_{S\times T}$-set, by Theorem  \ref{essential algeb}  $\overline{A\times B}\cap\beta\left(\mathcal{F}_{S\times T}\right)\neq\emptyset$. It can be observed that the set  $\widetilde{i}^{-1}\left[\left\{ \left(p,q\right)\right\} \right]\cap\beta\left(\mathcal{F}_{S\times T}\right)\neq\emptyset$ and
is a compact subsemigroup of $\beta\left(S\times T\right)$. Consequently $\overline{A\times B}\cap E\left(\beta\left(\mathcal{F}_{S\times T}\right)\right)\neq\emptyset$ and so  $A\times B$ is an essential $\mathcal{F}_{S\times T}$-set.
\end{proof}

Now, we will provide an elementary proof of the above theorem. To do so, we need the elementary characterization of essential $\mathcal{F}$-sets. In \cite{L}, J. Li established dynamical characterization of essential $\mathcal{F}$-sets for $S=\mathbb{N}$ and also found an elementary characterization using dynamical characterization, although elementary characterizations
of quasi central-sets and $C$-sets are known from \cite[Theorem 3.7]{HMS}
and \cite[Theorem 2.7]{HS09} respectively.  Quasi central
sets and $C$-sets are coming from by the setting of essential $\mathcal{F}$-set
and this fact confines that  essential $\mathcal{F}$-sets
might have elementary characterization for an arbitrary discrete semigroup.  Elementary characterization
of essential $\mathcal{F}$-sets from algebraic characterization are known from \cite[Theorem 5]{DDG} for arbitrary semigroup.

\begin{definition} Let $\omega$ be the first infinite ordinal and each ordinal
	indicates the set of all its predecessors. In particular, $0=\emptyset,$
	for each $n\in\mathbb{N},\:n=\left\{ 0,1,...,n-1\right\} $.
	\begin{itemize}	
        \item[(a)] If $f$ is a function and $dom\left(f\right)=n\in\omega$,
	then for all $x$, $f^{\frown}x=f\cup\left\{ \left(n,x\right)\right\} $.
		\item[(b)] Let $T$ be a set functions whose domains are members
	of $\omega$. For each $f\in T$, $B_{f}\left(T\right)=\left\{ x:f^{\frown}x\in T\right\}$.
\end{itemize}
\end{definition}	
	We get the following theorem from \cite[Theorem 5]{DDG}.

\begin{theorem}\label{Elementary F sets}
	Let $\left(S,\cdot\right)$ be a discrete semigroup, and assume that $\mathcal{F}$
	is a family of subsets of $S$ with the Ramsay property such that
	$\beta\left(\mathcal{F}\right)$ is a subsemigroup of $\beta S$.
	Let $A\subseteq S$. Statements (a), (b), and (c) are equivalent and
	are implied by statement (d). If $S$ is countable, then all the five
	statements are equivalent.
	
	\begin{itemize}
			
\item[(a)] $A$ is an essential $\mathcal{F}$-set.
	
\item[(b)] There is a non empty set $T$ of function such that
\begin{itemize}

	\item[(i)]  For all $f\in T$,$\text{domain}\left(f\right)\in\omega$
	and $rang\left(f\right)\subseteq A$;
	
	\item[(ii)]  For all $f\in T$ and all $x\in B_{f}\left(T\right)$,
	$B_{f^{\frown}x}\subseteq x^{-1}B_{f}\left(T\right)$; and
	
	\item[(iii)]  For all $F\in\mathcal{P}_{f}\left(T\right)$, $\bigcap_{f\in F}B_{f}(T)$
	is a $\mathcal{F}$-set.
	
	\end{itemize}

\item[(c)]  There is a downward directed family $\left\langle C_{F}\right\rangle _{F\in I}$
of subsets of $A$ such that
\begin{itemize}
	
\item[(i)] for each $F\in I$ and each $x\in C_{F}$ there exists
$G\in I$ with $C_{G}\subseteq x^{-1}C_{F}$ and

\item[(ii)] for each $\mathcal{F}\in\mathcal{P}_{f}\left(I\right),\,\bigcap_{F\in\mathcal{F}}C_{F}$
is a $\mathcal{F}$-set. 

\end{itemize}

\item[(d)] There is a decreasing sequence $\left\langle C_{n}\right\rangle _{n=1}^{\infty}$
of subsets of $A$ such that 

\begin{itemize}

\item[(i)] for each $n\in\mathbb{N}$ and each $x\in C_{n}$, there
exists $m\in\mathbb{N}$ with $C_{m}\subseteq x^{-1}C_{n}$ and 

\item[(ii)] for each $n\in\mathbb{N}$, $C_{n}$ is a $\mathcal{F}$-set.
\end{itemize}

\end{itemize}
\end{theorem}

Now we are in the position to present an elementary proof of our main result.

\begin{proof}\textbf{(Elementary proof of  Theorem \ref{main theorem})}:
Given that $A$ be an essential $\mathcal{F}_{S}$-set, then $\left\langle C_{F}\right\rangle _{F\in I}$
be as guaranteed by Theorem \ref{Elementary F sets} for $A$. It is also given that
$B$ be an essential $\mathcal{F}_{T}$-set, then $\left\langle D_{G}\right\rangle _{G\in J}$
be as guaranteed by Theorem \ref{Elementary F sets} for $B$. Direct $I\times J$ by
agreeing that $\left(F,G\right)\geq\left(F^{\prime},G^{\prime}\right)$
if and only if $F\geq F^{\prime}$ and $G\geq G^{\prime}$. We claim
that $\left\langle C_{F}\times D_{G}\right\rangle _{\left(F,G\right)\in I\times J}$
is as required by Theorem \ref{Elementary F sets} to show that $A\times B$ is an essential
$\mathcal{F}_{S\times T}$-set. Let $\left(F,G\right)\in I\times J$
and let $\left(x,y\right)\in C_{F}\times D_{G}$. Pick $H\in I$ and
$K\in J$ such that $\text{\ensuremath{C_{H}\subseteq x^{-1}C_{F}}}$
and $D_{K}\subseteq y^{-1}D_{G}$. Then $\left(H,K\right)\in I\times J$
, $C_{H}\times D_{K}\subseteq\left(x,y\right)^{-1}\left(C_{F}\times D_{G}\right)$
and consequently  $C_{H}\times D_{K}$ and $C_{F}\times D_{G}$ are $\mathcal{F}_{S\times T}$ by the given condition.
\end{proof}

\section{Dynamical proof}

In the previous section, we have proved our main result using the algebraic and elementary
characterizations of essential $\mathcal{F}$-set. To prove the same result using the dynamical characterizations of essential
$\mathcal{F}$-set, we first establish the dynamical characterizations
of essential $\mathcal{F}$-set by using results of \cite{J}.
Now we start with the following definitions.
\begin{definition}
Let $S$ be a discrete semigroup and $\mathcal{K}$ be a
filter on $S$.
\begin{itemize}
    \item[(a)] $\overline{\mathcal{K}}=\left\{ p\in\beta S:\mathcal{K}\subseteq p\right\} $.
    \item[(b)] $\mathcal{L}\left(\mathcal{K}\right)=\left\{ A\subseteq S:S\backslash A\notin\mathcal{K}\right\} $.
\end{itemize}
\end{definition}

The following is the definition of a dynamic system.
\begin{definition}
A pair $\left(X,\langle T_{s}\rangle_{s\in S}\right)$ is a dynamical
system if and only if it satisfies the following four conditions:
\begin{itemize}
    \item[(1)]  $X$ is a compact Hausdorff space.
    \item[(2)]  $S$ is a semigroup.
    \item[(3)] $T_{s}:X\rightarrow X$ is continuous for every $s\in S$.
    \item[(4)]  For every $s,t\in S$ we have $T_{st}=T_{s}\circ T_{t}$.
\end{itemize}

\end{definition}

Now, we define the jointly $\mathcal{K}$-recurrent.
\begin{definition}
Let $\left(X,\langle T_{s}\rangle_{s\in S}\right)$ be a dynamical
system, $x$ and $y$ points in $X$, and $\mathcal{K}$ a filter
on $S$. The pair $\left(x,y\right)$ is called jointly $\mathcal{K}$-
recurrent if and only if for every neighborhood $U$ of $y$ we have
$\left\{ s\in S:T_{s}\left(x\right)\in U\,\text{and}\,T_{s}\left(y\right)\in U\right\} \in\mathcal{L}\left(\mathcal{K}\right)$.
\end{definition}

We have the following important theorem \cite[Theorem 3.3]{J}
relating the above definitions.
\begin{theorem}\label{dynamic 2}
Let $\left(S,\cdot\right)$ be a semigroup, let $\mathcal{K}$ be
a filter on $S$ such that $\overline{\mathcal{K}}$ is a compact
subsemigroup of $\beta S$, and let $A\subseteq S$. Then $A$ is
a member of an idempotent in $\overline{\mathcal{K}}$ if and only
if there exists a dynamical system $\left(X,\langle T_{s}\rangle_{s\in S}\right)$
with points $x$ and $y$ in $X$ and there exists a neighbourhood
$U$ of $y$ such that the pair $\left(x,y\right)$ is jointly $\mathcal{K}$-
recurrent and $A=\left\{ s\in S:T_{s}\left(x\right)\in U\right\} $.
\end{theorem}

The following theorem is essential to apply the above theorem.
\begin{theorem}\label{theorem dynamics 1}
Let $\left(S,\cdot\right)$ be a discrete semigroup and $\mathcal{F}$ be a Ramsey family and $\mathcal{K}=\left\{ A\subseteq S:S\setminus A\text{ is not a }\mathcal{F}\text{-}\text{set}\right\} $.
Then $\mathcal{K}$ is a filter on $S$ with $\beta\left(\mathcal{F}\right)=\overline{\mathcal{K}}$. 
\end{theorem}

\begin{proof}
We have $\mathcal{K}$ is nonempty as $\emptyset\notin\mathcal{F}$,
and is closed under supersets. The family $\mathcal{K}$ is closed
under finite intersection follows from the fact that $\mathcal{F}$
is a Ramsey family and consequently $\mathcal{K}$ is a filter, $\mathcal{L\left(K\right)}=\left\{ A\subseteq S:A \text{ is a }\mathcal{F}\text{-}\text{set}\right\} $.
Hence it follows from \cite[Theorem 3.11 Page-59]{HS12} that $\beta\left(\mathcal{F}\right)=\overline{\mathcal{K}}$.
\end{proof}
Now, we are in the position to present a  dynamical characterization of essential
$\mathcal{F}$-set. 
\begin{theorem}\label{dynamical essential F set}
Let $\left(S,\cdot\right)$ be a discrete semigroup and $A\subseteq S$. Then
$A$ is an essential $\mathcal{F}$-set if and only if there exists
a dynamical system $\left(X,\langle T_{s}\rangle_{s\in S}\right)$
with points $x$ and $y$ in $X$ and there exists a neighbourhood
$U$ of $y$ such that 
\[
\left\{ s\in S:T_{s}\left(x\right)\in U\text{ and }T_{s}\left(y\right)\in U\right\} 
\]
 is a $\mathcal{F}$-set and $A=\left\{ s\in S:T_{s}\left(x\right)\in U\right\} $.
\end{theorem}

\begin{proof}
Let $\mathcal{K}=\left\{ A\subseteq S: S\setminus A\text{ is not a } \mathcal{F}\text{-}\text{set}\right\} $
and by  Theorem \ref{theorem dynamics 1}, $\beta\left(\mathcal{F}\right)=\overline{\mathcal{K}}$
and $\mathcal{L\left(K\right)}=\left\{ A\subseteq S:A\text{ is a }\mathcal{F}\text{-}\text{set}\right\} $. 
Hence we can apply  Theorem \ref{dynamic 2} to prove our statement.
\end{proof}
After getting a dynamical characterization of  essential $\mathcal{F}$-sets,
we can prove the main theorem in a dynamical approach.

\begin{proof}\textbf{ (Dynamical proof of Theorem \ref{main theorem}) }:
Given that $A$ is an essential $\mathcal{F}_{S}$-set, then there
exists a dynamical system $\left(X^{A},\langle T_{s}^{A}\rangle_{s\in S}\right)$
with points $x^{A}$ and $y^{A}$ in $X$ and there exists a neighbourhood
$U_{A}$ of $y^{A}$ such that 

\begin{itemize}
    \item $\left\{ s\in S:T_{s}^{A}\left(x^{A}\right)\in U^{A}\,\text{and}\,T_{s}^{A}\left(y^{A}\right)\in U^{A}\right\} $ is a $\mathcal{F}_{S}$-set and
    
    \item $A=\left\{ s\in S:T_{s}^{A}\left(x^{A}\right)\in U^{A}\right\} $.
\end{itemize}

It is also given that $B$ is an essential $\mathcal{F}_{T}$-set,
so that there exists a dynamical system $\left(X^{B},\langle T_{s}^{B}\rangle_{s\in S}\right)$
with points $x^{B}$ and $y^{B}$ in $X^{A}$ and there exists a neighbourhood
$U_{B}$ of $y^{B}$ such that 

\begin{itemize}
    \item $\left\{ t\in T:T_{t}^{B}\left(x^{B}\right)\in U^{B}\,\text{and}\,T_{t}^{B}\left(y^{B}\right)\in U^{B}\right\} $ is a $\mathcal{F}_{T}$-set and
    
    \item $B=\left\{ t\in T:T_{t}^{B}\left(x^{B}\right)\in U^{B}\right\} $.
\end{itemize}

To show that $A\times B$ is an essential $\mathcal{F}_{S\times T}$-set,
consider the dynamical system $\left(X^{A}\times X^{B},\langle T_{s}^{A}\times T_{t}^{B}\rangle_{\left(s,t\right)\in S\times T}\right)$. The set
$\left\{ \left(s,t\right)\in S\times T:T_{s}^{A}\times T_{t}^{B}\left(x^{A},x^{B}\right)\in U^{A}\times U^{B}\,\text{ and }\,T_{s}^{A}\times T_{t}^{B}\left(y^{A},y^{B}\right)\in U^{A}\times U^{B}\right\}$\\

	  = $\left\{ \left(s,t\right)\in S\times T:\left(T_{s}^{A}\left(x^{A}\right),T_{t}^{B}\left(x^{B}\right)\right)\in U^{A}\times U^{B}\,\text{ and }\,\left(T_{s}^{A}\left(y^{A}\right),T_{t}^{B}\left(y^{B}\right)\right)\in U^{A}\times U^{B}\right\}$\\
   
	=  $\left\{ s\in S:T_{s}^{A}\left(x^{A}\right)\in U^{A}\,\text{ and }\,T_{s}^{A}\left(y^{A}\right)\in U^{A}\right\} \times\left\{ t\in T:T_{t}^{B}\left(x^{B}\right)\in U^{B}\,\text{ and }\,T_{t}^{B}\left(y^{B}\right)\in U^{B}\right\}$\\
is a $\mathcal{F}_{S\times T}$-set.\\
Now, we have
$\left\{ \left(s,t\right)\in S\times T:T_{s}^{A}\times T_{t}^{B}\left(y^{A},y^{B}\right)\in U^{A}\times U^{B}\right\}$\\

	 =  $\left\{ \left(s,t\right)\in S\times T:\left(T_{s}^{A}\left(y^{A}\right),T_{t}^{B}\left(y^{B}\right)\right)\in U^{A}\times U^{B}\right\}$\\
  
	=  $\left\{ s\in S:T_{s}^{A}\left(x^{A}\right)\in U^{A}\right\} \times\left\{ t\in T:T_{t}^{B}\left(x^{B}\right)\in U^{B}\right\}$\\
 
 = $A\times B$.\\
Hence $A\times B$ is an essential $\mathcal{F}_{S\times T}$-set by Theorem \ref{dynamical essential F set}.

\end{proof}

\noindent\textbf{Acknowledgements.} The author would like to convey his heartfelt thanks to the referee for giving valuable comments in the previous draft of this article.

\end{document}